\documentclass[12pt]{article}

\usepackage{alltt, amsmath, amssymb,amsthm, amsfonts, graphicx, verbatim}

\newtheorem{theorem}{Theorem}
\newtheorem{lemma}[theorem]{Lemma}

\newtheorem{proposition}[theorem]{Proposition}

\newtheorem{example}[theorem]{Example}

\newtheorem{remark}[theorem]{Remark}

\title{Stabilization  of stochastic approximation by step size adaptation}

\date{}

\begin{document}

\author{Sameer Kamal\footnote{School of Technology and Computer Science, Tata Institute of Fundamental Research,
Homi Bhabha Road, Mumbai 400005, India. E-mail:
sameer.kamal@gmail.com. This research was supported in part by an
Infosys Fellowship.}}

\maketitle

\vspace{.5in}

\noindent \textbf{Abstract:} A scheme for stabilizing stochastic
approximation iterates by adaptively scaling the step sizes is
proposed and analyzed. This scheme leads to the same limiting
differential equation as the original scheme and therefore has the
same limiting behavior, while avoiding the difficulties associated
with projection schemes. The proof technique requires only that the
limiting o.d.e. descend a certain Lyapunov function outside an
arbitrarily large bounded set.

 \vspace{.5in}


\section{Introduction}

Stochastic approximation was originally introduced in \cite{robbins}
as a scheme for finding zeros of a nonlinear function under noisy
measurements. It has since become one of the main workhorses of
statistical computation, signal processing, adaptive schemes in
control engineering and artificial intelligence, economic models,
and so on. See \cite{benveniste}, \cite{borkar}, \cite{chen},
\cite{duflo}, \cite{kushner} for some recent texts that give an
extensive account. One of the successful approaches for its
convergence analysis has been the `o.d.e.\ approach' of \cite{dere},
\cite{ljung} which treats it as a noisy discretization of an
ordinary differential equation (o.d.e.) with slowly decreasing step
sizes. The convergence analysis is usually of the form: if the
iterates remain stable, i.e., a.s.\ bounded, then they converge
a.s.\ to a set predicted by the o.d.e.\ analysis. Stability tests
that establish a.s. boundedness are typically geared for specific
applications and require stringent assumptions on the `drift' term.
See, e.g., \cite{abounadi}, \cite{borkarmeyn}, \cite{Tsi} for some
recent stability tests motivated by reinforcement learning
applications, that crucially use resp.\ long term stability w.r.t.\
initial data, exact linear growth, or contraction-like properties
for the drift. There does not seem to be a broad enough test to cover
a reasonably generic class of stochastic approximation algorithms.\\

An alternative to establishing a priori stability is to force it by
suitably modifying the algorithm, the most popular modification being
to project it onto a bounded set every time it exits from the same
\cite{KushCl}, \cite{chen}. This, however, is not without its pitfalls.
One major problem is that the projection operation can introduce spurious equilibria.
Another is that the choice of the bounded set in question needs to be carefully done,
in particular it should include the desired asymptotic limit (point or set) which is
usually not known a priori.\\

Motivated by this, we propose and analyze a different scheme for
stabilizing the iterates, viz., an adaptation of step sizes that
controls the growth of the iterates without affecting their
asymptotic behavior. This amounts to scaling the step sizes
appropriately when the iterates are sufficiently far away from the
origin. In fact, one can argue that at most a finite random number
of steps differ from the original scheme.\\

Another offshoot of our analysis is that instead of requiring the
o.d.e. to descend the Lyapunov function everywhere where the
function isn't at its minimum, we only require it to do so outside a sphere of arbitrarily large radius.
While this  is hardly surprising, the fact does not seem to have been formally recorded in literature.

\section{Preliminaries}

Throughout this article we allow the letter $c$ to denote a possibly
different constant in different places.\\

 Consider the $\mathbb R^d$-valued stochastic approximation iterates
 \begin{equation}
 x_{n+1} = x_{n} + a(n)[h(x_n)+M_{n+1}], \label{SA}
 \end{equation}
and their `o.d.e.' limit
\begin{equation}
\dot{x}(t) = h(x(t)). \label{ode}
\end{equation}
Let $W(\cdot):\mathbb R^d\rightarrow[0, \infty)$ be a continuously
differentiable Lyapunov function. We make the following assumptions
regarding $h(\cdot)$, $a(n)$, $M_{n+1}$, and $W(\cdot)$

\begin{itemize}
\item[\textbf{(A1)}] $h(\cdot)$ is locally Lipschitz.

\item[\textbf{(A2)}] Step size assumptions.
\begin{itemize}
\item[(i)] $\sum_n a_n=\infty$.
\item[(ii)] $\sum_n a_n^2<\infty$.
\end{itemize}

\item[\textbf{(A3)}]
Martingale difference assumptions.
\begin{itemize}
\item[(i)] $(M_n)$ is a martingale difference sequence w.r.t. the
filtration $(\mathcal{F}_n)$ where $\mathcal{F}_n=\sigma(x_0,
M_1,\ldots, M_n)$. Thus, $E[M_{n+1}|\mathcal{F}_n]=0$ a.s. for all
$n\geq0$.
\item[(ii)] $M_n$ is square integrable for all $n\geq 0$ and there
exists a locally bounded and measurable function $f(\cdot):\mathbb R^d\rightarrow
[0, \infty)$ such that
\[
E[\|M_{n+1}\|^2|\mathcal{F}_n]\leq f(x_n)\text{ a.s. }
\]
\end{itemize}
\item[\textbf{(A4)}] Lyapunov function assumptions.
\begin{itemize}
\item[(i)] $W(x)\geq 0$ for all $x\in \mathbb R^d$ and $W(x)\rightarrow\infty$ as $\|x\|\rightarrow\infty$.
\item[(ii)] There exists a positive integer, say $M$, such that
\[
h(x)\cdot\nabla W(x) < 0 \text{ whenever } W(x) \geq M.
\]
\end{itemize}
\end{itemize}

We next define a generalization of the iteration scheme~(\ref{SA}).
First, choose a positive integer $N$, with $M<N\leq\infty$, such
that there is a finite positive constant $c_N$ satisfying
\begin{equation}
\label{c}
 c_N > 1 \bigvee \left(\sup_{y\in \bar H^N\backslash
H^M}\frac{\|h(y)\|^2 + f(y)}{W(y)}\right).
\end{equation}
At least for finite $N$, assumptions \textbf{(A1)} and
\textbf{(A3)}(ii) guarantee such a choice for $c_N$. Having chosen a
suitable $N$, choose a locally bounded measurable function
$g(\cdot):\mathbb R^d\rightarrow \mathbb R$ such that
\begin{equation}
\label{g}
 g(y)> 1 \bigvee \left(\mathbb I\left\{W(y)>N\right\}\sqrt{\frac{\|h(y)\|^2 +
f(y)}{W(y)}}\right).
\end{equation}
Again, assumptions \textbf{(A1)} and \textbf{(A3)}(ii) guarantee
such a choice for $g(\cdot)$. We thus have, for some suitable $N$,
possibly infinite, the following inequality
\begin{equation}
\label{wgc} c_N W(y) > \frac{\|h(y)\|^2 + f(y)}{g(y)^2}\text{ if }
W(y)\geq M.
\end{equation}
Having chosen $g(\cdot)$, consider the iterates $\{y_n\}$ generated
by
\begin{equation}
y_{n+1} = y_{n} + a^\omega(n)[h(y_n)+M_{n+1}], \label{MSA}
\end{equation}
where
\begin{equation}
\label{newstep} a^\omega(n):= a(n)/g(y_n).
\end{equation}
This is a generalization of the original iteration scheme~(\ref{SA})
since the step size $a^\omega(n)$ is now an $\mathcal
F_n$-measurable random step size. We note that by our choice
\begin{itemize}
\item $g(\cdot)$ is a locally bounded function, and
\item $g(y)\geq 1$ for all $y\in \mathbb R^d$.
\end{itemize}

\begin{remark}
\label{remark} By choosing $N$ large enough we can ensure $g(y)=1$
for $y$ in an arbitrarily large sphere around the origin. If
$c_\infty<\infty$, we can choose $N=\infty$, in which case $g(y)=1$
for all $y \in \mathbb R^d$ and we recover the original
scheme~(\ref{SA}).
\end{remark}

\begin{remark}
\label{randomstep}
 Since $g(y)\geq 1$ for all $y\in \mathbb R^d$, it
follows from assumption \textbf{(A2)}(ii) that the random step
sizes satisfy
\[
\sum_n a^\omega(n)^2<\infty \text{ a.s. }
\]
\end{remark}

\section{A test for stability}

Let $m$ be an arbitrary positive integer, $m>M$. Define the level
set
\[
H^m := \{x:W(x)< m\},
\]
and let $\bar H^m$ denote the closure of $H^m$. Since
$h(x)\cdot\nabla W(x) < 0$ whenever $W(x) \geq M$, we get
\[
\dot{W}(x):=h(x)\cdot\nabla W(x) < 0 \text{ for } x\in \bar
H^m\backslash H^M.
\]

As $\bar H^m\backslash H^M$ is a compact set, and $\dot{W}(\cdot)$
is a continuous function, there must exist a negative constant $c$
such that
\begin{equation}
\sup_{x\in \bar H^m\backslash H^M}\dot{W}(x) \leq c < 0. \label{W
dot}
\end{equation}

Fix some $T>0$. Note that \textbf{(A1)} and \textbf{(A4)} ensure the
well-posedness of the o.d.e. given by~(\ref{ode}) for $t \geq 0$.
Let $y^u(t)$ be the o.d.e. trajectory starting from $u$. Thus, $\dot
y^u(t)= h(y^u(t))$ for $t \geq 0$, and $y^u(0) = u$. Choose a
positive but arbitrarily small $\epsilon_m$ satisfying
\[
\epsilon_m \leq 1 \wedge \inf \left\{\left|W(u)- W(v)\right|: u,v\in
\bar H^m\backslash H^M \text{ and } v=x^u(T)\right\}.
\]
Note that $\epsilon_m>0$ is possible because of~(\ref{W dot}). Given
$\epsilon_m$, choose a positive but arbitrarily small $\delta_m$
such that:
\[
\text{ if } u,v\in \bar H^m \ \text{and} \ \|u-v\|<\delta_m, \
\text{ then } \ |W(u)-W(v)|<\epsilon_m/2.
\]
Note that $\delta_m>0$ is possible  because $W(\cdot)$ is a
continuous function and $\bar H^m$ is a compact set.
\begin{remark}
\label{arbitrarilysmall} The fact that both $\epsilon_m$ and
$\delta_m$ can be chosen positive but arbitrarily small will prove
crucial later.
\end{remark}

Let $n_0 \geq 0$. Given $n_i(\omega)$, define $n_{i+1}(\omega)$ as
\[
n_{i+1}(\omega):=\inf\left\{n>n_i(\omega):\sum_{n_i(\omega)}^n
a^\omega(i)\geq T \right\}.
\]

Consider the $\delta_m$-neighbourhood of $H^m$,
\[
N^{\delta_m}(H^m) :=\left\{x:\inf_{y\in H^m}
\|x-y\|<\delta_m\right\}.
\]
Note that $\mathbb I\left\{y_n \in N^{\delta_m}
\left(H^m\right)\right\}a^\omega(n)M_{n+1}$ is a martingale
difference term. Since $N^{\delta_m} \left(H^m\right)$ is a bounded
set, and $f(\cdot)$ is locally bounded, it follows from assumption
\textbf{(A3)}(ii) and Remark~\ref{randomstep} that
 \begin{eqnarray}
\nonumber &&\sum_n \mathbb E\left[\left(\|\mathbb I\left\{y_n \in
N^{\delta_m} \left(H^m\right)\right\}a^\omega(n)M_{n+1}
\|\right)^2|\mathcal
F_n\right] \\
\nonumber &\leq& \left(\sup_{y \in N^{\delta_m}
\left(H^m\right)}f(\|y\|)\right)\times \sum_n a^\omega(n)^2\\
\nonumber &<& \infty \text{ a.s.} \label{equation}\\
\end{eqnarray}
This leads to:
\begin{lemma}
\label{lemma1} Assume \textbf{(A2)}--\textbf{(A4)}. For any positive
integer $m>M$ we have:
\[
\sum_i \mathbb I\left\{y_n \in N^{\delta_m}
\left(H^m\right)\right\}a^\omega(n)M_{n+1} \text{ converges a.s.}
\]
\end{lemma}
\begin{proof}
This is immediate from (\ref{equation}) and the convergence theorem
for square-integrable martingales, Theorem 3.3.4, p.\ 53,  of
\cite{borkprob}.
\end{proof}

From Lemma~\ref{lemma1} it follows that almost surely there exists
an $N(\omega, m)$ such that if $n_0(\omega)\geq N(\omega, m)$ then
\begin{equation}
\sup_q\left\|\sum_{n_0(\omega)}^q\mathbb I\left\{y_n \in
N^{\delta_m} \left(H^m\right)\right\} a^\omega(n)M_{n+1}\right\| <
\frac{\delta_m}{2\exp{(KT)}}.\label{constrained convergence}
\end{equation}

\begin{remark}
\label{convergenceremark} Note that Lemma~\ref{lemma1} guarantees
the convergence of the martingale $\sum_i \mathbb
I\{\cdots\}a^\omega(n)M_{n+1}$ while not saying anything about
$\sum_n a^\omega(n)$.  Since the martingale converges, there must
exist an $N(\omega, m)$ satisfying~(\ref{constrained convergence})
even if $\sum_n a^\omega(n)<\infty$. That $\sum_n
a^\omega(n)=\infty$ a.s.\ needs a proof. In what follows, we give a
sufficient condition for stability and show that it is also
sufficient for $\sum_n a^\omega(n)=\infty$ a.s.
\end{remark}

Let $K$ be the Lipschitz constant of $h(\cdot)$ on
$N^{\delta_m}\left(H^m\right)$. Without loss of generality we assume
that $N(\omega, m)$ is large enough that if $n_0(\omega) \geq
N(\omega, m)$. Then
\begin{equation}
K\left(\sup_{y\in N^{\delta_m}
\left(H^m\right)}\|h(y)\|\right)\left(\sum_{n_0(\omega)}^\infty
a^\omega(n)^2\right)<\frac{\delta_m}{2\exp{(KT)}}. \label{tail}
\end{equation}

\begin{lemma}
Assume \textbf{(A1)}--\textbf{(A4)}. Let $m$ be a positive integer
with $m > M$. Let $n_0(\omega)$ satisfy $n_0(\omega)\geq N(\omega,
m) $. Under this base condition, the following inductive step holds:
if $n_i(\omega)<\infty$ and $y_{n_i}(\omega) \in H^m$, then
\begin{enumerate}
\item  $y_j(\omega)\in N^{\delta_m}\left(H^m\right)$ a.s. for $n_i(\omega) \leq j \leq n_{i+1}(\omega)$,
\item $n_{i+1}(\omega)<\infty$ a.s., and
\item Almost surely, either
\begin{itemize}
\item $W\left(y_{n_{(i+1)}}(\omega)\right)<W\left(y_{n_i}(\omega)\right)-\frac{\epsilon_m} 2$, or
\item $y_{n_{(i+1)}}(\omega)\in N^{\delta_m}(H^M)$.
\end{itemize}
In particular, in either case, $y_{n_{(i+1)}}(\omega) \in H^m$ a.s.
\end{enumerate}
\label{bounded}
\end{lemma}
\begin{proof}
We first show by induction that $y_j(\omega)\in
N^{\delta_m}\left(H^m\right)$ for $n_i(\omega) \leq j \leq
n_{i+1}(\omega)$. By assumption, $y_{n_i}(\omega) \in H^m \subset
N^{\delta_m}\left(H^m\right)$. Fix $j$ in the range $n_i(\omega) < j
\leq n_{i+1}(\omega)$. Assume $y_k(\omega)\in
N^{\delta_m}\left(H^m\right)$ for $n_i(\omega) \leq k \leq j-1$. We
need to show that $y_j(\omega)\in N^{\delta_m} \left(H^m\right)$.
 If
\begin{equation}
K\left(\sup_{n_i\leq k\leq
j-1}\|h(y_k)\|\right)\left(\sum_{n_i(\omega)}^{j-1}a^\omega(n)^2\right)<\frac{\delta_m}{2\exp{(KT)}},
\label{first condition}
\end{equation}
and
\begin{equation}
\sup_{n_i\leq k\leq j-1}\left\|\sum_{n_i(\omega)}^{k}
a^\omega(n)M_{n+1}\right\|<\frac{\delta_m}{2\exp{(KT)}},
\label{second condition}
\end{equation}
then by a standard application of the Gronwall inequality (see,
e.g., Lemma~2.1 in \cite{borkar}) $y_j(\omega)$ will satisfy
\begin{equation}
\left\|y_j(\omega)-y^{y_{n_i}(\omega)}\left(\sum_{n_i(\omega)}^{j-1}a^\omega(n)\right)\right\|<\delta_m.
\label{deviation}
\end{equation}
From the assumption that $n_i(\omega)\geq N(\omega, m)$ it follows
that (\ref{constrained convergence}) and (\ref{tail}) hold. These
equations, coupled with the assumption that $y_k(\omega)\in
N^{\delta_m}\left(H^m\right)$ for $n_i(\omega) \leq k \leq j-1$,
imply (\ref{first condition}) and (\ref{second condition}), which in
turn imply (\ref{deviation}). Since the o.d.e. trajectory will
always be in $H^m$ if it starts there, (\ref{deviation}) implies
\[
y_j(\omega)\in N^{\delta_m} \left(H^m\right).
\]
Induction now proves the first claim.\\

For the second claim we give a proof by contradiction. Consequently,
assume that $n_{i+1}(\omega)=\infty$. The first claim, which has
already been proved, now gives $y_j(\omega)\in
N^{\delta_m}\left(H^m\right)$ a.s. for $n_i(\omega) \leq j \leq
\infty$. Therefore, since $g(\cdot)$ is a locally bounded function,
we get $\sup_{j \geq n_i(\omega)} g(y_j(\omega))<\infty$. By
assumption \textbf{(A2)}(i) this gives
\[
\sum_{j = n_i(\omega)}^{\infty}a^\omega(j)\geq \frac{\sum_{j =
n_i(\omega)}^{\infty}a(j)}{\sup_{j \geq n_i(\omega)} g(y_j(\omega))}
= \infty.
\]
Since $n_{i+1}(\omega)=\infty$ requires $\sum_{j =
n_i(\omega)}^{\infty}a^\omega(j)\leq T$, we get the required
contradiction. Thus $n_{i+1}(\omega)<\infty$ a.s.\\

We turn to the final claim. Let
$z=y^{y_{n_i}(\omega)}\left(\sum_{n_i}^{n_{(i+1)}-1}a^\omega(n)\right)$,
the o.d.e. trajectory after time $\approx T$ starting from
$y_{n_i}(\omega)$. Since the o.d.e. starts in $H^m$, it remains in
$H^m$. There are two cases to consider.
\begin{itemize}
\item If $z \in H^m\backslash H^M$, the definition of $\epsilon_m$ implies that $W(z)
\leq W\left(y_{n_i}(\omega)\right) - \epsilon_m$.
Since~(\ref{deviation}) holds for $j=n_{(i+1)}(\omega)$, we have
$\|y_{n_{(i+1)}}(\omega) - z\|<\delta_m$. From the definition of
$\delta_m$ it follows that
$W\left(y_{n_{(i+1)}}(\omega)\right)<W(z)+\epsilon_m/2$. We get
$W\left(y_{n_{(i+1)}}(\omega)\right)<W(y_{n_i}(\omega))-\frac
{\epsilon_m} 2$. In particular, $y_{n_{(i+1)}}(\omega)\in H^m$.
\item  If $z\in H^M$, then, since $\|y_{n_{(i+1)}}(\omega) - z\|<\delta_m$, we get $y_{n_{(i+1)}}(\omega)\in N^{\delta_m}(H^M)$.
Since $N^{\delta_m}(H^M)\subset H^{M+\frac {\epsilon_m} 2}$ and
$m>M+1/2$, we get $y_{n_{(i+1)}}\in H^m$.
\end{itemize}
The proof is complete.
\end{proof}

Define the stopping times
\[
\tau_k^m(\omega) := \inf\{n\geq k: W(y_n(\omega))<m\}.
\]
The next result establishes the fact that if $W(y_n(\omega))<m$ for
infinitely many $n$, then almost surely the iterates converge to
$\bar H^M$.

\begin{proposition}
Assume \textbf{(A1)}--\textbf{(A4)}. For any arbitrary $m> M$, if
$\tau_k^m(\omega)<\infty$ for all $k$, then $y_n(\omega) \rightarrow
\bar H^M$ a.s. \label{convergence}
\end{proposition}

\begin{proof}
Assume $\tau_k^m(\omega)<\infty$ for all $k$. From the definition of
$\tau_k^m$ this implies that given any $k$ there exists an $n$ with
$n\geq k$ such that $y_n(\omega)\in H^m$. In other words, the
iterates are in $H^m$ infinitely often. By
Remark~\ref{convergenceremark} there exists an  $N(\omega, m)$
satisfying~(\ref{constrained convergence}). Since the iterates are
in $H^m$ infinitely often there exists an $n_0 > N(\omega, m)$ such
that $y_{n_0} \in H^m$. From Lemma~\ref{bounded}
 we know that if $n_i(\omega)<\infty$ then almost surely $n_{i+1}(\omega)<\infty$.
 By induction it follows that $n_i(\omega)<\infty$ a.s. for all $i\in \mathbb Z^+$.
 Invoking Lemma~\ref{bounded} again, we get that either
$W\left(y_{n_{(i+1)}}(\omega)\right)<W\left(y_{n_i}(\omega)\right)-\frac
{\epsilon_m} 2$, or $y_{n_{(i+1)}}(\omega)\in N^{\delta_m}(H^M)$.
Since $W(\cdot)$ cannot keep decreasing by $\epsilon_m/2$ forever,
it follows that for some $i$, $y_{n_i}(\omega)\in
N^{\delta_m}(H^M)$. Note that $N^{\delta_m}(H^M)\subset H^{M+\frac
{\epsilon_m} 2}$. Consequently, if $y_{n_i}(\omega)\in
N^{\delta_m}(H^M)$ then $W\left(y_{n_i}(\omega)\right) < M + \frac
{\epsilon_m} 2$ and so $y_{n_{(i+1)}}(\omega)\in N^{\delta_m}(H^M)$.
It follows that the iterates $y_{n_i}(\omega)$ will eventually get
trapped in $N^{\delta_m}(H^M)$. Once the iterates $y_{n_i}(\omega)$
are trapped in $N^{\delta_m}(H^M)\subset H^{M+\frac {\epsilon_m}
2}$, the o.d.e. starting from $y_{n_i}(\omega)$ will remain in
$H^{M+\frac {\epsilon_m} 2}$. It follows that once the iterates
$y_{n_i}(\omega)$  are trapped in $N^{\delta_m}(H^M)$, the
intermediate iterates $y_j(\omega)$, $n_i(\omega) < j <
n_{(i+1)}(\omega)$ will get trapped in $N^{\delta_m}\left(H^{M+\frac
{\epsilon_m} 2}\right)$. Since both $\epsilon_m$ and $\delta_m$ were
chosen arbitrarily small positive quantities (see
Remark~\ref{arbitrarilysmall}), the result follows.
\end{proof}

Consider two statements of stability: first
\begin{equation}
\label{gst} y_n(\omega)\rightarrow \bar H^M \text{ a.s. }
\end{equation}
and second, for every positive integer $k\geq 0$,
\begin{equation}
\label{wst} y_{n\wedge\tau_k^M}(\omega)\rightarrow \bar H^M \text{
a.s. }
\end{equation}
The next result establishes the equivalence of the two stability
statements.

\begin{lemma}
\label{equivalent} Under assumptions \textbf{(A1)}--\textbf{(A4)},
the two stability statements~(\ref{gst}) and~(\ref{wst}) are
equivalent.
\end{lemma}
\begin{proof}
Clearly~(\ref{gst}) implies~(\ref{wst}). For the converse,
assume~(\ref{wst}). We need to show that
\[
\mathbb P\left[y_{n\wedge\tau_k^M}(\omega)\rightarrow \bar H^M \
\forall k \text{ and } y_n(\omega)\not\rightarrow \bar H^M
\right]=0.
\]
Fix an $m>M$. Let $\omega$ be such that
$y_{n\wedge\tau_k^M}(\omega)\rightarrow \bar H^M \ \forall k$ and
$y_n(\omega)\not\rightarrow \bar H^M$. Since
$y_n(\omega)\not\rightarrow \bar H^M$, by
Proposition~\ref{convergence} there exists a $k$ such that
$\tau_k^m(\omega)=\infty$ a.s. For this choice of $k$, since $m>M$,
it follows that $\tau_k^M(\omega)=\infty$ a.s. Thus for this $k$,
$y_{n\wedge\tau_k^M}(\omega)\rightarrow \bar H^M$ reduces to
$y_n(\omega)\rightarrow \bar H^M$ a.s. The result follows.
\end{proof}

On the basis of Lemma~\ref{equivalent}, we get the following test
for stability: for every $k$, if
$y_{n\wedge\tau_k^M}(\omega)\rightarrow \bar H^M$ a.s. then $\sup_n
\|y_n\|< \infty$ a.s. Note that it
does not require the o.d.e. to descend the Lyapunov function
inside the arbitrarily large set $H^M$. In the next section we give
a sufficient condition for this stability test.

\section{A sufficient condition for stability}

In this section we show that assumption \textbf{(A5)} below is
sufficient for stability.
\begin{itemize}
\item[\textbf{(A5)}] Let the $W(\cdot)$ of \textbf{(A4)} be twice continuously differentiable such that
all second order derivatives of $W(\cdot)$ are bounded in absolute
value by a constant.
\end{itemize}

We start with a few lemmas.

\begin{lemma}
\label{sufficient} Assume \textbf{(A1)}--\textbf{(A5)}. For any
positive integer $k$, and for any $\mathcal F_k$-measurable set $A$,
if $\mathbb E[W\left(y_k(\omega); A\right)]<\infty$ then
\[
\sup_{n\geq k}\mathbb
E\left[W\left(y_{n\wedge\tau_k^M}(\omega)\right); A\right]<\infty.
\]
\end{lemma}
\begin{proof}
We have
\[
y_{(n+1)\wedge\tau_k^M} = y_{n\wedge\tau_k^M} + a^\omega(n)\mathbb
 I\{\tau_k^M>n\}\left[h\left(y_{n\wedge\tau_k^M}\right) + M_{n+1}\right].
\]

Doing a Taylor expansion and using the fact that the second order
space derivatives of $W(\cdot)$ are bounded, we get
\begin{eqnarray}
\nonumber && W\left(y_{(n+1)\wedge\tau_k^M}\right)\\
\nonumber &\leq& W\left(y_{n\wedge\tau_k^M}\right) +
a^\omega(n)\mathbb I\{\tau_k^M>n\}\nabla
W\left(y_{n\wedge\tau_k^M}\right) \cdot\left[h\left(y_{n\wedge\tau_k^M}\right) + M_{n+1}\right]\\
\nonumber && \ + ca^\omega(n)^2
I\{\tau_k^M>n\}\left\|h\left(y_{n\wedge\tau_k^M}\right) +
M_{n+1}\right\|^2.
\end{eqnarray}
Since $\mathbb I\{\tau_k^M>n\}\nabla
W\left(y_{n\wedge\tau_k^M}\right) \cdot
h\left(y_{n\wedge\tau_k^M}\right)\leq 0$ and $\mathbb
E\left[h\left(y_{n\wedge\tau_k^M}\right)\cdot M_{n+1}|\mathcal
F_n\right]=0$, we get
\begin{eqnarray}
\nonumber && \mathbb
E\left[W\left(y_{(n+1)\wedge\tau_k^M}\right)|\mathcal F_n\right] \\
\nonumber &\leq& W\left(y_{n\wedge\tau_k^M}\right)+ ca^\omega(n)^2
\left(\mathbb E\left[\left.\mathbb
I\{\tau_k^M>n\}\left(\left\|h\left(y_{n\wedge\tau_k^M}\right)\right\|^2+\|M_{n+1}\|^2\right)\right|\mathcal
F_n\right]\right).
\end{eqnarray}

From~(\ref{wgc}) and the definition of $a^\omega(n)$, it follows
that
\begin{eqnarray}
\nonumber E\left[W\left(y_{(n+1)\wedge\tau_k^M}\right)|\mathcal F_n\right] &\leq& W\left(y_{n\wedge\tau_k^M}\right) + ca(n)^2\cdot c_NW\left(y_{n\wedge\tau_k^M}\right)\\
\nonumber &\leq& (1+ca(n)^2)W\left(y_{n\wedge\tau_k^M}\right)\\
\nonumber &\leq& \exp{(ca(n)^2)}W\left(y_{n\wedge\tau_k^M}\right)
\end{eqnarray}
For $n\geq k$, integrating gives
\[
E\left[W\left(y_{(n+1)\wedge\tau_k^M}\right); A\right] \leq
\exp{\left(c\sum_{i=k}^\infty a(i)^2\right)}\mathbb
E[W\left(y_k(\omega); A\right)] <\infty.
\]
The result follows.
\end{proof}

The next lemma is independent of assumption \textbf{(A5)} and
requires only assumptions \textbf{(A1)}--\textbf{(A4)} for its
proof.

\begin{lemma}
\label{lemmaforeveryk} Assume \textbf{(A1)}--\textbf{(A4)}. Let $k$
be an arbitrary positive integer. Let $A$ be an arbitrary $\mathcal
F_k$-measurable set. If
\[
\sup_{n\geq k}\mathbb
E\left[W(y_{n\wedge\tau_k^M}(\omega)); A\right]<\infty
\]
then
\[
\mathbb P \left[A \bigcap
\left(y_{n\wedge\tau_k^M}(\omega)\not\rightarrow \bar
H^M\right)\right]=0.
\]
\end{lemma}

\begin{proof}
Assume $y_{n\wedge\tau_k^M}(\omega)\not\rightarrow \bar H^M$.
Clearly, this implies that $\tau_k^M(\omega)=\infty$ and so
$y_{n\wedge\tau_k^M}(\omega)=y_n(\omega)$. It follows that
$y_n(\omega)\not\rightarrow \bar H^M$. Now, let $u$ be an arbitrary
integer, $u>M$. By Proposition~\ref{convergence}, since
$y_n(\omega)\not\rightarrow \bar H^M$, there exists an integer $l$,
$l\geq k$,  such that $\tau_l^u(\omega)=\infty$ a.s. It follows that
\[
\left\{A\bigcap\left(y_{n\wedge\tau_k^M}(\omega)\not\rightarrow \bar
H^M\right)\right\} = \bigcup_l
\left\{A\bigcap\left(y_{n\wedge\tau_k^M}(\omega)\not\rightarrow \bar
H^M \text{ and } \tau_l^u(\omega)=\infty\right)\right\}.
\]
Since $\{\tau_l^u(\omega)=\infty\}\subset
\{\tau_{l+1}^u(\omega)=\infty\}$ for all $l$, it follows that there
exists a positive integer $L \geq k$ such that
\[
\mathbb
P\left[A\bigcap\left(y_{n\wedge\tau_k^M}(\omega)\not\rightarrow \bar
H^M \text{ and } \tau_L^u(\omega)=\infty\right)\right]> \frac 1 2
\times \mathbb
P\left[A\bigcap\left(y_{n\wedge\tau_k^M}(\omega)\not\rightarrow \bar
H^M\right)\right].
\]

Combining everything we get the following inequalities\\
\begin{minipage}{\linewidth}
\begin{eqnarray}
\nonumber && \sup_{n\geq k}\mathbb
E\left[W(y_{n\wedge\tau_k^M}(\omega)); A\right]\\
\nonumber &\geq& \mathbb
E\left[W(y_{L\wedge\tau_k^M}(\omega)); A\right]\\
\nonumber  &\geq& u \times \mathbb
P\left[A\bigcap\left(\tau_k^M(\omega)=\infty \text{ and } \tau_L^u(\omega)=\infty\right)\right]\\
\nonumber &\geq& u \times \mathbb
P\left[A\bigcap\left(y_{n\wedge\tau_k^M}(\omega)\not\rightarrow \bar
H^M
\text{ and } \tau_L^u(\omega)=\infty\right)\right]\\
\nonumber &\geq& \frac u 2 \times \mathbb
P\left[A\bigcap\left(y_{n\wedge\tau_k^M}(\omega)\not\rightarrow \bar
H^M\right)\right].
\end{eqnarray}
\end{minipage}
Since $u$ is arbitrary and $\sup_{n\geq k}\mathbb
E\left[W(y_{n\wedge\tau_k^M}(\omega)); A\right]<\infty$, the result
follows.
\end{proof}

\begin{lemma}
\label{stronglemma} Assume \textbf{(A1)}--\textbf{(A5)}. For $k$ an
arbitrary positive integer, we have
\[
\mathbb P \left[y_{n\wedge\tau_k^M}(\omega)\not\rightarrow \bar
H^M\right]=0.
\]
\end{lemma}
\begin{proof}
Define $A^l:=\{\omega: W(y_k(\omega))<l\}$. Clearly $A^l$ is
$\mathcal F_k$-measurable and $\mathbb E[W\left(y_k(\omega)\right);
A^l]<l<\infty$. It follows from Lemma~\ref{sufficient} that
\[
\sup_{n\geq k}\mathbb
E\left[W\left(y_{n\wedge\tau_k^M}(\omega)\right); A^l\right]<\infty.
\]
Lemma~\ref{lemmaforeveryk} now gives us
\[
\mathbb P \left[A^l \bigcap
\left(y_{n\wedge\tau_k^M}(\omega)\not\rightarrow \bar
H^M\right)\right]=0.
\]
 Since $\mathbb P \left[\bigcup_l A^l\right]=1$ it follows that
$\mathbb P \left[y_{n\wedge\tau_k^M}(\omega)\not\rightarrow \bar
H^M\right]=0$
\end{proof}

We now give our main results and a couple of examples.

\begin{theorem}
\label{strongresult} Under assumptions \textbf{(A1)}--\textbf{(A5)},
\[
y_n(\omega)\rightarrow \bar H^M \text{ a.s.}
\]
In particular, $\sup_n \|y_n\|< \infty$ a.s.
\end{theorem}
\begin{proof}
The result follows from Lemma~\ref{equivalent} and
Lemma~\ref{stronglemma}.
\end{proof}

The next results establish that the iterates $(y_n)$ indeed capture
bahaviour as time goes to infinity.

\begin{proposition}
Under assumptions \textbf{(A1)}--\textbf{(A5)}, almost surely
$a^\omega(n)=a(n)$ for all except finitely many $n$. In particular,
\[
\sum_n a^\omega(n) = \infty \text{ a.s. }
\]
\end{proposition}
\begin{proof}
By Theorem~\ref{strongresult}, $y_n(\omega)\rightarrow \bar H^M$
a.s. Since $g(y)=1$ for $y\in H^N$, and $N>M$, it follows that
$g(y_n)=1$ for all except finitely many $n$.
\end{proof}

Finally, following~\cite{benaim} (see also \cite{borkar}, Chapter 2), we get
\begin{theorem}
\label{theorem}
 Under assumptions \textbf{(A1)}--\textbf{(A5)}, the
iterates $(y_n)$ converge a.s.\ to an internally chain transitive
set of the o.d.e.
\end{theorem}

We also get a condition for the convergence of the iterates $(x_n)$
obtained  by the original iteration scheme as given by~(\ref{SA}).
\begin{theorem}
Under assumptions \textbf{(A1)}--\textbf{(A5)}, if
\[
\sup_{x\in \mathbb R^d}\frac{\|h(x)\|^2 +
f(x)}{1\wedge W(x)}= c_\infty<\infty
\]
then the original iterates $(x_n)$ converge a.s.\ to an internally
chain transitive set of the o.d.e.
\end{theorem}
\begin{proof}
By Remark~\ref{remark} we can set $N=\infty$ in~(\ref{c}).
Now~(\ref{g}) gives $g(x)=1$ for all $x\in\mathbb R^d$.
Equation~(\ref{wgc}) continues to hold with $c_\infty$ in place of
$c_N$. The choice of $g(\cdot)$ gives $a^\omega(n)=a(n)$ for all
$n$, or $x_n(\omega)=y_n(\omega)$ for all $n$. The result now
follows from Theorem~\ref{theorem}.
\end{proof}

\begin{example}
 Consider the scalar iteration
\begin{displaymath}
x_{n+1} = x_n - a(n)x_n\exp{(|x_n|)}(1 + \xi_{n+1}),
\end{displaymath}
where $\{\xi_n\}$ are i.i.d.\ $N(0, 1)$ (say). Here $W(x) = x^2$ and
$g(x) = O(\exp{(|x|)})$ will do.
\end{example}

\begin{example}
\label{example} Consider the scalar iteration~(\ref{SA}) with
bounded $h(\cdot)$ satisfying
\[
\lim_{x\uparrow\infty}h(x) = -\lim_{x\downarrow -\infty}h(x) = -1,
\]
with $\{M_n\}$ i.i.d.\ uniform on $[-1, 1]$. Then $W(x) = x^2$ and
$g(x) \equiv 1$ will do. In particular, there is no need to
adaptively scale the step sizes.
\end{example}

Note that neither of these two examples, even the apparently simple
Example~\ref{example}, is covered by the tests of \cite{abounadi},
\cite{borkarmeyn}, \cite{Tsi}.\\

\textbf{Acknowledgements:} The author would like to thank Prof.\ V.\
S.\ Borkar for suggesting this problem and for his comments on an
earlier draft which included, in particular, the idea of using an
adaptive scheme.

\end{document}